\documentclass[a4paper,10pt]{amsart}

\usepackage[utf8]{inputenc}
\usepackage[english]{babel}
\usepackage[T1]{fontenc}
\usepackage{amsfonts}
\usepackage{amsmath}
\usepackage{amsthm}
\usepackage{amssymb}
\usepackage{hyperref}
\usepackage{syntonly}
\usepackage{mathtools}
\usepackage{algorithm}
\usepackage{algpseudocode}
\usepackage{tikz}
\usepackage{dsfont}
\usepackage{braket}
\usepackage{faktor}
\usepackage{array}
\usepackage{lscape}
\usepackage{rotating}
\usepackage{epigraph}
\usepackage{subfig}
\usepackage{indentfirst}
\usepackage{booktabs}
\usepackage{multirow}
\usepackage{caption}
\usepackage{tabu}
\usepackage{tabularx}
\usepackage{xcolor}
\usepackage{pgf}
\usepackage{tikz}
\usepackage{caption}
\usepackage{color}
\usepackage{enumitem}
\usepackage{mathtools}
\usepackage{float}
\usepackage{graphicx}
\usepackage[displaymath, tightpage]{preview}
\usepackage[toc,page]{appendix}

\usetikzlibrary{matrix}
\usetikzlibrary{arrows}
\hypersetup{
    colorlinks,
    linkcolor={red!50!black},
    citecolor={blue!50!black},
    urlcolor={blue!80!black}
}

\usepackage[colorinlistoftodos, textwidth=4cm, shadow,disable]{todonotes}

\theoremstyle{plain}
\newtheorem{theorem}{Theorem}[section]

\theoremstyle{definition}

\theoremstyle{remark}

\begin{document}

\title[Euler-Poincar\'{e} characteristic of parallel arrangements]{On the Euler-Poincar\'{e} characteristic of parallel toric arrangements}

\author[Elia Saini]{Elia Saini}
\address{(Elia Saini) Independent researcher, Via San Rocco 40/A, IT-23036, Teglio (SO).}
\email{saini.elia@outlook.com}

\begin{abstract}
Toric arrangements of maximal rank have been studied by the author in a paper that shows how the complement manifold of these arrangements is diffeomorphic to that of centered ones.
In this work we turn our attention to toric arrangements of rank one, namely parallel toric arrangements. We prove, by means of basic arguments of cohomology theory, that the Euler-Poincar\'{e} characteristic of the complement manifold of parallel toric arrangements can be computed in terms of those of the complement manifolds of the singular subtori that compose
the arrangement.
\end{abstract}

\maketitle

 \section*{Introduction}

Toric arrangements represent a widely analyzed category of hypersurface configurations that play a significant role across multiple mathematical disciplines, including combinatorics, algebraic geometry, and topology.
Drawing upon the foundational investigations of Lefschetz \cite{L24} and Deligne \cite{D71}, these structures extend the concept of hyperplane arrangements to the framework of analyzing the complements of normal crossing divisors within smooth projective varieties.

A highly productive line of inquiry is dedicated to characterizing the topological and combinatorial invariants associated with these specific systems. The initial determination of the Betti numbers for the complement manifold of a toric configuration was accomplished by Looijenga \cite{L93} through the application of principles from sheaf theory. Subsequently, an explicit description of the complex cohomology ring was established in \cite{DP10}.
The intricate combinatorial framework of toric arrangements and their profound connections to arithmetic matroids have been investigated by several researchers, such as Br\"{a}nd\' {e}n, D'Adderio, Lenz, and Moci, across a sequence of papers including \cite{BM14}, \cite{DM13}, \cite{L17}, and \cite{M12}.

The fundamental problem of defining the homotopy type of the complement manifold for these configurations was initially addressed by Moci and Settepanella in \cite{MS11}. Later, d'Antonio and Delucchi \cite{dD12, dD15} utilized the toric complex introduced in \cite{MS11} to outline the fundamental group and establish a minimality theorem for toric arrangements that specialize to a real configuration. A direct outcome of this property is that the integer cohomology for this category of arrangements is torsion-free.

In the investigations carried out by Randell \cite{Ran02} alongside Dimca and Papadima \cite{DP03}, it is demonstrated that the complement manifold of any hyperplane arrangement is invariably minimal. A comparable property does not extend to generic hypersurfaces. Indeed, examining the complement of the plane cusp serves as a sufficient counterexample, as indicated in \cite{Ran02}. Consequently, the motivation to investigate the homotopy type of toric arrangement complements stems from the need to clarify which criteria govern the validity of these minimality properties within broader hypersurface configurations.

Lastly, by refining the analysis of the toric complex and the underlying poset of layers, Callegaro and Delucchi provided a formulation of the integer cohomology for the complement manifold of centered toric arrangements in \cite{CD17}. To differentiate which algebraic and topological invariants are dictated strictly by combinatorics, Pagaria introduced an example in \cite{Pag19} featuring two distinct toric arrangements that share an identical poset of layers yet exhibit different integer cohomologies. Meanwhile, De Concini and Gaiffi demonstrated in \cite{DCG21} that the rational cohomology ring of toric configurations is intrinsecally combinatorial.

Using basic techniques from linear algebra, the author has shown in \cite{Sai24} that for toric arrangements that have an associated matrix of maximal rank it is possible to perform a sequence of diffeomorphic changes of coordinates that transform these arrangemenets into centered ones, allowing the application of the results on  complexified arrangements developed by d'Antonio and Delucchi in \cite{dD15}.

In this work we focus our attention on parallel toric arrangements, a specific class of toric arrangement with associated matrix of rank one. With basic techniques from cohomology theory, we prove that the Euler-Poincar\'{e} characteristic of the complement manifold of these toric arrangements can be computed in terms of those of the complement manifolds of the singular subtori that compose the arrangement itself.

    \smallskip
   \noindent \textbf{Overview.} In Section \ref{Section1} we recall some basic definitions and results on toric arrangements and a definition of Euler-Poincar\'{e} characteristic in terms of de Rham cohomology. In Section \ref{Section2} we show our result on the Euler-Poincar\'{e} characteristic of parallel toric arrangements.

\section{Basics}\label{Section1}
The aim of this introductory section is to refresh some basic notions about toric arrangements, in order to set up notations and fix the proper frame for the further part of this paper.
For a comprehensive theory of toric arrangements we suggest reviewing the paper \cite{DP10} as well as the survey \cite{DP11}. On the other hand, for a thorough analysis of the homotopy type of the complement manifold of toric arrangements we recall the articles \cite{dD15} and \cite{Sai24}.

Finally, we point out the seminal work of De Concini and Procesi \cite{DP11} to get a better understanding of the inner and surprising connections between toric arrangements, partition functions and box splines.

\subsection{Toric arrangements}\label{TA}
A \textit{toric arrangement} in $(\mathbb{C}^{\ast})^{n}$ is a finite collection $\mathcal{A}=\{T_{1},\ldots,T_{m}\}$ of subspaces of
$(\mathbb{C}^{\ast})^{n}$, called \textit{subtori}, of the form

\begin{equation*}
 T_{i}=\left\lbrace
        (z_{1},\ldots,z_{n})\in(\mathbb{C}^{\ast})^{n}
        \left|
         z_{1}^{a_{i,1}}\cdots z_{n}^{a_{i,n}}=\alpha_{i}
        \right.
       \right\rbrace
\end{equation*}

\noindent where $a_{i,j}\in\mathbb{Z}$ and $\alpha_{i}\in\mathbb{C}^{\ast}$ as $1\leq i\leq m$ and $1\leq j\leq n$. The \textit{complement manifold} $M(\mathcal{A})$ is the complement of the union of the subtori $H_{i}$ in $(\mathbb{C}^{\ast})^{n}$, that is, the topological space

\begin{equation*}
 M(\mathcal{A})=(\mathbb{C}^{\ast})^{n}\setminus\bigcup_{i=1}^{m}H_{i}
\end{equation*}

\noindent The matrix \textit{associated} to the toric arrangement $\mathcal{A}$ is the integer matrix  of $m$ rows and $n$ columns $\tilde{M}_{\mathcal{A}}\in M_{m,n}(\mathbb{Z})$ defined by setting

\begin{equation*}
 \tilde{M}_{\mathcal{A}}=\left(
                  \begin{array}{ccc}
                   a_{1,1} & \cdots & a_{1,n} \\
                   \vdots  & \vdots & \vdots  \\
                   a_{m,1} & \cdots & a_{m,n} \\
                  \end{array}
                 \right)
\end{equation*}

\noindent The \textit{rank} $\operatorname{rk}(\mathcal{A})$ of the toric arrangement $\mathcal{A}$ is then the rank of the associated matrix $\tilde{M}_{\mathcal{A}}$. We say that $\mathcal{A}$ is \textit{parallel} if $T_{i}\cap T_{j}=\emptyset$ for any pair of indices $1\leq i,j\leq m$ with $i\neq j$.

\subsection{Euler-Poincar\'{e} characteristic}\label{EPC}
For the sake of simplicity, since we are working with smooth manifolds, throughout this work we are dealing with de Rham cohomology. Thanks to de Rham theorem (compare \cite[Chapters 17 and 18]{Lee13}), in the context of differentiable manifolds this cohomology theory is equivalent to the singular one. Therefore it is possible to provide topological definitions in terms of de Rham cohomology. Then, the \textit{Euler-Poincar\'{e} characteristic} $\chi_{M}$ of a differentiable manifold $M$ of real dimension $d$ can be defined as the alternating sum

\begin{equation*}
 \chi_{M}=\sum_{k=0}^{d}(-1)^{k}\operatorname{dim}H^{k}(M)
\end{equation*}

\noindent where $H^{k}(M)$ is the $k\text{-th}$ de Rham cohomology group of $M$. We remark that from the definition of de Rham cohomology it follows that $H^{k}(M)=0$ as $k\leq-1$ and $k\geq d+1$. Moreover, we recall that a complex manifold of complex dimension $n$ is also a differentiable manifold of real dimension $2n$.

\section{Results}\label{Section2}
In this section we prove our main result. More precisely, we are showing that the Euler-Poincar\'{e} characteristic of the complement manifold of a parallel toric arrangement can be solely computed from those of the complement manifolds of the singular subtori that compose the arrangement.

\begin{theorem}\label{main}
 Let $m,n\geq1$ and let $\mathcal{A}=\{T_{1},\ldots,T_{m}\}$ be a toric arrangement in $(\mathbb{C}^{\ast})^{n}$ with complement manifold $M(\mathcal{A})$. Let $M(\mathcal{A}_{i})$ be the complement of $T_{i}$ in $(\mathbb{C}^{\ast})^{n}$ as $1\leq i\leq m$. If $\mathcal{A}$ is parallel, then the Euler-Poincar\'{e} characteristic of $M(\mathcal{A})$ can be computed in terms of those of $M(\mathcal{A}_{i})$, that is $\chi_{M(\mathcal{A})}=\sum_{i=1}^{m}\chi_{M(\mathcal{A}_{i})}$.
\end{theorem}

\begin{proof}
 We prove our statement by induction on the number $m$ of subtori that compose the arrangement $\mathcal{A}$. The base case $m=1$ is trivial. For the inductive step, let us assume our statement is true for every parallel toric arrangement that consists of $m$ subtori and let us consider a parallel toric arrangement $\mathcal{A}=\{T_{1},\ldots,T_{m+1}\}$ in 
$(\mathbb{C}^{\ast})^{n}$ that is made of $m+1$ subtori. Set $U=M(\mathcal{A}_{m+1})$ and $V=M(\tilde{\mathcal{A}})$ where $\mathcal{A}_{m+1}=\{T_{m+1}\}$ and $\tilde{\mathcal{A}}=\{T_{1},\ldots,T_{m}\}$. From the definition of complement manifold of an arrangement we have
\begin{equation*}
 U\cap V=M(\mathcal{A})
\end{equation*}

\noindent Moreover, since the toric arrangement $\mathcal{A}$ is parallel we also can deduce that
\begin{equation*}
 U\cup V=(\mathbb{C}^{\ast})^{n}
\end{equation*}

\noindent Indeed, we have
\begin{equation*}
\begin{aligned}
 U\cup V & = M(\mathcal{A}_{m+1})\cup M(\tilde{\mathcal{A}})= \\
         & = \left[(\mathbb{C}^{\ast})^{n}\setminus T_{m+1}\right]\cup\left[(\mathbb{C}^{\ast})^{n}\setminus\bigcup_{j=1}^{m}T_{j}\right]= \\
         & = (\mathbb{C}^{\ast})^{n}\setminus\left(T_{m+1}\cap\bigcup_{j=1}^{m}T_{j}\right)= \\
         & = (\mathbb{C}^{\ast})^{n}\setminus\bigcup_{j=1}^{m}\left(T_{m+1}\cap T_{j}\right)= \\
         & = (\mathbb{C}^{\ast})^{n} \\
\end{aligned}
\end{equation*}

\noindent Now, let us consider the Mayer-Vietoris sequence for the de Rham cohomology associated to the open cover $\{U,V\}$ of $(\mathbb{C}^{\ast})^{n}$, that is,
\begin{equation*}
 \cdots\rightarrow H^{k}((\mathbb{C}^{\ast})^{n})\rightarrow H^{k}(M(\mathcal{A}_{m+1}))\oplus H^{k}(M(\tilde{\mathcal{A}}))\rightarrow H^{k}(M(\mathcal{A}))\rightarrow\cdots
\end{equation*}

\noindent From the fact that this sequence is exact and $H^{k}(M)=0$ as $k\leq-1$ and $k\geq d+1$ we have
\begin{equation*}
\chi_{(\mathbb{C}^{\ast})^{n}}+\chi_{M(\mathcal{A})}=\chi_{M(\mathcal{A}_{m+1})}+\chi_{M(\tilde{\mathcal{A}})}
\end{equation*} 

\noindent The Euler-Poincar\'{e} characteristic is a homotopic invariant, thus
\begin{equation*}
 \chi_{(\mathbb{C}^{\ast})^{n}}=\chi_{(S^{1})^{n}}
\end{equation*}

\noindent Exploiting the K\"{u}nneth formula
\begin{equation*}
 H^{k}((S^{1})^{n})=\mathbb{R}^{\binom{n}{k}}
\end{equation*}

\noindent so that
\begin{equation*}
 \chi_{(S^{1})^{n}}=\sum_{k=0}^{n}(-1)^{k}\binom{n}{k}=0
\end{equation*}

\noindent Hence,
\begin{equation*}
 \chi_{M(\mathcal{A})}=\chi_{M(\mathcal{A}_{m+1})}+\chi_{M(\tilde{\mathcal{A}})}
\end{equation*}  

\noindent By inductive step
\begin{equation*}
 \chi_{M(\tilde{\mathcal{A}})}=\sum_{j=1}^{m}\chi_{M(\mathcal{A}_{j})}
\end{equation*} 

\noindent and then our statement is finally proved since we have
\begin{equation*}
\chi_{M(\mathcal{A})}=\chi_{M(\mathcal{A}_{m+1})}+\sum_{j=1}^{m}\chi_{M(\mathcal{A}_{j})}=\sum_{j=1}^{m+1}\chi_{M(\mathcal{A}_{j})}
\end{equation*}  

\end{proof}

%

%

\smallskip
\noindent \textbf{Acknowledgments.} The author would like to thank Matthias Lenz for the very helpful support provided throughout all
the time this project has been developed.

\bibliographystyle{amsalpha}

\bibliography{Full_revised}

\end{document}